\documentclass[12pt]{amsart}
\usepackage{graphicx}
\usepackage[headings]{fullpage}
\usepackage{amssymb,epic,eepic,epsfig,amsbsy,amsmath,amscd}
\numberwithin{equation}{section}
                        \theoremstyle{plain}
\usepackage{mathrsfs}

\newcommand\no[1]{}

\newtheorem{theorem}{Theorem}[section]
\newtheorem{thm}{Theorem}
\newtheorem{lemma}[theorem]{Lemma}

\newtheorem{proposition}[theorem]{Proposition}

\theoremstyle{definition}
\newtheorem{remark}[theorem]{Remark}

\def\BC{\mathbb C}

\def\BZ{\mathbb Z}

\def\la{\langle}
\def\ra{\rangle}

\DeclareMathOperator{\tr}{\mathrm tr}

\def\be { \begin{equation} }
\def\ee { \end{equation} }

\begin{document}

\title[T.A.P. of genus one two-bridge knots]{Twisted Alexander polynomials of genus one two-bridge knots}

\author[Anh T. Tran]{Anh T. Tran}
\address{Department of Mathematical Sciences, University of Texas at Dallas, Richardson, TX 75080, USA}
\email{att140830@utdallas.edu}

\begin{abstract}
 Morifuji \cite{Mo-08} computed the twisted Alexander polynomial of twist knots for nonabelian representations. In this paper we compute the twisted Alexander polynomial and the Reidemeister torsion of genus one two-bridge knots, a class of knots which includes twist knots. As an application, we give a formula for the Reidemeister torsion of the 3-manifold obtained by a Dehn surgery on a genus one two-bridge knot.
\end{abstract}

\thanks{2010 {\em Mathematics Classification:} Primary 57N10. Secondary 57M25.\\
{\em Key words and phrases: Dehn surgery, nonabelian representation, Reidemeister torsion, twist knot.}}

\maketitle

\section{Introduction}

The twisted Alexander polynomial, a generalization of the
Alexander polynomial, was introduced by Lin \cite{Li} for knots in $S^3$ and by Wada \cite{Wada94-1} for
finitely presented groups. It was interpreted in terms of Reidemeister torsion by
Kitano \cite{Ki} and Kirk-Livingston \cite{KL}. Twisted Alexander polynomials have been extensively studied in the past ten years by many authors, see the survey papers \cite{FV, Mo-15} and references therein. 

In \cite{Mo-08} Morifuji  computed the twisted Alexander polynomial of twist knots for nonabelian representations. In this paper we will generalize his result to genus one two-bridge knots. In a related direction, Kitano \cite{Ki2015} gave a formula for the Reidemeister torsion of the 3-manifold obtained by a Dehn surgery on the figure eight knot. In \cite{Tr} we generalized his result to twist knots. In this paper we will also compute the Reidemeister torsion of the 3-manifold obtained by a Dehn surgery on a genus one two-bridge knot. 

Let $J(k,l)$ be the link in Figure 1, where $k,l$ denote 
the numbers of half twists in the boxes. Positive (resp. negative) numbers correspond 
to right-handed (resp. left-handed) twists. 
Note that $J(k,l)$ is a knot if and only if $kl$ is even. It is known that the set of all genus one two-bridge knots  is the same as the set of all the knots $J(2m,2n)$ with $mn \not= 0$, see e.g. \cite{BZ}. The knots $J(2,2n)$ are known as twist knots. For more information on $J(k,l)$, see \cite{HS}.

\begin{figure}[th]
\centerline{\psfig{file=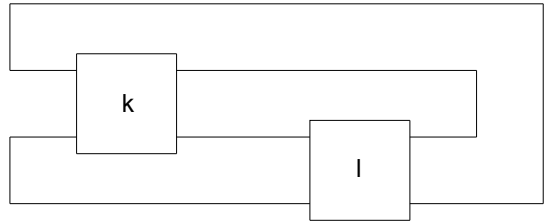,width=3.5in}}
\vspace*{8pt}
\caption{The link $J(k,l)$. }
\end{figure} 

From now on we fix $K=J(2m,2n)$ with $mn \not=0$. The knot group of $K$ has a presentation $\pi_1(K)= \la a, b \mid w^na=bw^n \ra$ where $a,b$ are meridians and $w=(ba^{-1})^m(b^{-1}a)^m$. A representation $\rho:  \pi_1(K) \to SL_2(\BC)$ is called nonabelian if 
the image of $\rho$ is a nonabelian subgroup of $SL_2(\BC)$. Suppose $\rho: \pi_1(K) \to SL_2(\BC)$ is a nonabelian representation. Up to conjugation, we may assume that $$\rho(a) = \left[ \begin{array}{cc}
s & 1 \\
0 & s^{-1} \end{array} \right] \quad \text{and} \quad \rho(b) = \left[ \begin{array}{cc}
s & 0 \\
2-y & s^{-1} \end{array} \right]$$ where $s \not=0$ and $y \not= 2$ satisfy the Riley equation $\phi_K(s,y)=0$, see \cite{Ri, Le}. The polynomial $\phi_K(s,y)$ will be computed explicitly in Section \ref{nab}. Note that $y=\tr \rho(ab^{-1})$. 

Let $S_k(v)$ be the Chebychev polynomials of the second kind defined by $S_0(v)=1$, $S_1(v)=v$ and $S_{k}(v) = v S_{k-1}(v) - S_{k-2}(v)$ for all integers $k$. 

Let $x :=\tr \rho(a)=s + s^{-1}$ and $z:= \tr \rho(w)=2+(y-2)(y+2-x^2)S^2_{m-1}(y)$. 

\begin{thm}
\label{main} 
Suppose $\rho: \pi_1(K) \to SL_2(\BC)$ is a nonabelian representation. Then the twisted Alexander polynomial of $K$ is given by 
\begin{eqnarray*}
\Delta_{K,\rho}(t) &=& (t+t^{-1}-x) \left( \frac{S_{m}(y)-S_{m-2}(y)-2}{y-2} \right) \left( \frac{S_{n}(z)-S_{n-2}(z)-2}{z-2} \right) \\
&& + \,  xS_{m-1}(y) S_{n-1}(z).
\end{eqnarray*}
\end{thm}

\begin{thm}
\label{main1} 
Suppose $\rho: \pi_1(K) \to SL_2(\BC)$ is a nonabelian representation. If $x \not= 2$ then the Reidemeister torsion of $K$ is given by $$\tau_{\rho}(K)=
(2-x) \left( \frac{S_{m}(y)-S_{m-2}(y)-2}{y-2} \right) \left( \frac{S_{n}(z)-S_{n-2}(z)-2}{z-2} \right) + xS_{m-1}(y) S_{n-1}(z).$$
\end{thm}

Now let $M$ be the 3-manifold obtained by a $\frac{p}{q}$-surgery on the genus one two-bridge knot $K$. The fundamental group $\pi_1(M)$ has a presentation
$$\pi_1(M) = \la a, b \mid w^n a = bw^n, a^p\lambda^q=1 \ra,$$
where $\lambda$ is the canonical longitude corresponding to the meridian $\mu=a$.

\begin{thm} 
\label{main2}
Suppose $\rho: \pi_1(K) \to SL_2(\BC)$ is a nonabelian representation which extends to a representation $\rho: \pi_1(M) \to SL_2(\BC)$. If $x \not\in \{0,2\}$ then the Reidemeister torsion of $M$ is given by   
\begin{eqnarray*}
\tau_{\rho}(M) &=&
\Big\{ (2-x) \Big( \frac{S_{m}(y)-S_{m-2}(y)-2}{y-2} \Big) \Big( \frac{S_{n}(z)-S_{n-2}(z)-2}{z-2} \Big) \\
&& + \, xS_{m-1}(y) S_{n-1}(z) \Big\} \Big( \frac{4-x^2 + (y+2-x^2)(y-2)S^2_{m-1}(y)}{x^2 (y-2)^2 S^2_{m-1}(y)}\Big).
\end{eqnarray*}
\end{thm}

\begin{remark} (1) Theorem \ref{main} generalizes the formula for the twisted Alexander polynomial of twist knots by Morifuji \cite{Mo-08}. 

(2) Theorem \ref{main2} generalizes the formulas for the Reidemeister torsion of the 3-manifold obtained by a $\frac{p}{q}$-surgery on the figure eight knot by Kitano \cite{Ki2015} and on twist knots by the author \cite{Tr}.
\end{remark}

The paper is organized as follows. In Section \ref{nab} we give a formula for the Riley polynomial of a genus one two-bridge knot, and compute the trace of a canonical longitude. In Section \ref{section-R} we review the twisted Alexander polynomial and the Reidemeister torsion of a knot. We prove Theorems \ref{main}, \ref{main1} and \ref{main2} in Section \ref{app}.

\section{Nonabelian representations} 

\label{nab}

In this section we give a formula for the Riley polynomial of a genus one two-bridge knot.  We also compute the trace of a canonical longitude.

\subsection{Chebyshev polynomials}

Recall that $S_k(v)$ are the Chebychev polynomials defined by $S_0(v)=1$, $S_1(v)=v$ and $S_{k}(v) = v S_{k-1}(v) - S_{k-2}(v)$ for all integers $k$. The following lemma is elementary. We will use it many times without referring to it.

\begin{lemma} \label{chev} One has $S^2_k(v) - v S_k(v) S_{k-1}(v) + S^2_{k-1}(v)=1.$
\end{lemma}

Let $P_k(v) := \sum_{i=0}^k S_i(v)$. The next two lemmas are proved in \cite{Tr}.

\begin{lemma}
\label{P_k}
One has $P_k(v) = \frac{S_{k+1}(v)-S_{k}(v)-1}{v-2}.$
\end{lemma}

\begin{lemma}
\label{formulas}
Suppose $V = \left[ \begin{array}{cc}
a & b \\
c & d \end{array} \right] \in SL_2(\BC)$. Then 
\begin{eqnarray}
V^k &=& \left[ \begin{array}{cc}
S_{k}(v) - d S_{k-1}(v) & b S_{k-1}(v) \\
c S_{k-1}(v) & S_{k}(v) - a S_{k-1}(v) \end{array} \right], \label{power}\\
\sum_{i=0}^k V^i &=& \left[ \begin{array}{cc}
P_{k}(v) - d P_{k-1}(v) & b P_{k-1}(v)\\
c P_{k-1}(v) & P_{k}(v) - a P_{k-1}(v) \end{array} \right], \label{sum-power}
\end{eqnarray}
where $v:= \tr V = a+d$. Moreover, one has
\begin{equation} 
\label{det-sum}
\det \left( \sum_{i=0}^k V^i \right) = \frac{S_{k+1}(v) - S_{k-1}(v)-2}{v-2}.
\end{equation}
\end{lemma}

\subsection{The Riley polynomial} Recall that $K = J(2m,2n)$.  The knot group of $K$ has a presentation $\pi_1(K) = \la a, b \mid w^n a = b w^n\ra$ where $a,b$ are meridians and $w=(ba^{-1})^m(b^{-1}a)^m$, see \cite{HS}.  Suppose $\rho: \pi_1(K) \to SL_2(\BC)$ is a nonabelian representation. Up to conjugation, we may assume that $$\rho(a) = \left[ \begin{array}{cc}
s & 1 \\
0 & s^{-1} \end{array} \right] \quad \text{and} \quad \rho(b) = \left[ \begin{array}{cc}
s & 0 \\
2-y & s^{-1} \end{array} \right]$$ where $s \not=0$ and $y \not= 2$ satisfy the Riley equation $\phi_K(s,y)=0$.

We now compute $\phi_K(s,y)$. Since $\rho(ba^{-1})=\left[ \begin{array}{cc}
1 & -s \\
s^{-1}(2-y) & y-1 \end{array} \right]$ and $y=\tr \rho(ba^{-1})$, by Lemma \ref{formulas} we have 
$$\rho((ba^{-1})^m)=\left[ \begin{array}{cc}
S_m(y)-(y-1)S_{m-1}(y) & -sS_{m-1}(y) \\
s^{-1}(2-y)S_{m-1}(y) & S_m(y)-S_{m-1}(y) \end{array} \right].$$
Similarly $$\rho((b^{-1}a)^m)=\left[ \begin{array}{cc}
S_m(y)-(y-1)S_{m-1}(y) & s^{-1}S_{m-1}(y) \\
s(y-2)S_{m-1}(y) & S_m(y)-S_{m-1}(y) \end{array} \right].$$
Hence $\rho(w)=\rho((ba^{-1})^m(b^{-1}a)^m)=\left[ \begin{array}{cc}
W_{11} & W_{12} \\
(2-y)W_{12} & W_{22} \end{array} \right]$ where
\begin{eqnarray*}
W_{11} &=& S^2_{m}(y) + (2-2y)S_m(y)S_{m-1}(y) + (1+2s^2-2y-s^2y+y^2)S^2_{m-1}(y),\\
W_{12} &=& (s^{-1}-s)S_m(y)S_{m-1}(y) + (s^{-1}+s-s^{-1}y)S^2_{m-1}(y),\\
W_{22} &=& S^2_{m}(y) - 2S_m(y)S_{m-1}(y) + (1+2s^{-2}-s^{-2}y)S^2_{m-1}(y).
\end{eqnarray*}

Let $z=\tr \rho(w)$. Since $S^2_{m}(y) - yS_m(y)S_{m-1}(y) +S^2_{m-1}(y)=1$ (by Lemma \ref{chev}), we have 
\begin{eqnarray*}
z = W_{11} + W_{22} &=& 2(S^2_{m}(y) - yS_m(y)S_{m-1}(y) +S^2_{m-1}(y))\\
&&+ \, (2s^2+2s^{-2}-2y-s^2y-s^{-2}y+y^2)S^2_{m-1}(y)\\
&=& 2 + (y-2)(y-s^2-s^{-2})S^2_{m-1}(y).
\end{eqnarray*}
By Lemma \ref{formulas} we have $\rho(w^n) = \left[ \begin{array}{cc}
S_n(z)-W_{22} \, S_{n-1}(z) & W_{12} \, S_{n-1}(z) \\
(2-y)W_{12} \, S_{n-1}(z) & S_n(z)-W_{11} \, S_{n-1}(z) \end{array} \right]$. Hence
$$\rho(w^na-bw^n)=\left[ \begin{array}{cc}
0 & \phi_K(s,y) \\
(2-y)\phi_K(s,y) & 0 \end{array} \right]$$ where $\phi_K(s,y)$
\begin{eqnarray*}
 &=& S_n(z) - \big\{ (s-s^{-1})W_{12}+W_{22} \big\}  S_{n-1}(z)\\
&=& S_n(z) - \big\{ S^2_{m}(y) - (s^2+s^{-2})S_m(y)S_{m-1}(y) + (1+s^2+s^{-2}-y)S^2_{m-1}(y) \big\} S_{n-1}(z)\\
&=& S_n(z) - \big\{ 1  + (y-s^2-s^{-2})S_{m-1}(y) \big( S_m(y)- S_{m-1}(y) \big) \big\} S_{n-1}(z).
\end{eqnarray*}

\begin{remark}
Similar formulas for $\phi_K(s,y)$ were already obtained in \cite{MPL, MT}.
\end{remark}

\subsection{Trace of the longitude} By \cite{HS} the canonical longitude of $K$ corresponding to the meridian $\mu=a$ is $\lambda=\overleftarrow{w}^n w^n$, where $\overleftarrow{w}$ is the word in the letters $a,b$ obtained by writing $w$ in the reversed order. We now compute its trace. This computation will be used in the proof of Theorem \ref{main2}.

Let $\alpha = 1  + (y-s^2-s^{-2})S_{m-1}(y) \big( S_m(y)- S_{m-1}(y)$.

\begin{lemma} \label{a^2}
One has
$$\alpha^2 - z \alpha +1 = (y-s^2-s^{-2})S^2_{m-1}(y) \big( 2-s^2-s^{-2} + (y-s^2-s^{-2})(y-2)S^2_{m-1}(y) \big).$$
\end{lemma}

\begin{proof}
By a direct calculation we have
\begin{eqnarray*}
\alpha^2 - z \alpha +1 &=& (y-s^2-s^{-2})S^2_{m-1}(y)\big\{ 2-y + (y-s^2-s^{-2}) \\
&&\big( S^2_{m}(y) - yS_m(y)S_{m-1}(y) + (y-1)S^2_{m-1}(y) \big) \big\}.
\end{eqnarray*}
The lemma follows, since $S^2_{m}(y) - yS_m(y)S_{m-1}(y) + S^2_{m-1}(y)=1$.
\end{proof}

\begin{lemma} 
\label{S^2}
One has
$$S^2_{n-1}(z)= \big\{ (y-s^2-s^{-2})S^2_{m-1}(y) \left( 2-s^2-s^{-2} + (y-s^2-s^{-2})(y-2)S^2_{m-1}(y) \right) \big\}^{-1}.$$
\end{lemma}

\begin{proof}
 Since $s \not= 0$ and $y \not= 2$ satisfy the Riley equation $\phi_K(s,y)=0$, we have $S_n(z) =\alpha S_{n-1}(z)$. Hence
$$1 = S^2_n(z) - z S_n(z) S_{n-1}(z) + S^2_{n-1}(z) = (\alpha^2 - z \alpha +1)S^2_{n-1}(z).$$
The lemma then follows from Lemma \ref{a^2}.
\end{proof}

\begin{proposition}
\label{longitude}
One has $$\tr\rho(\lambda)  = 2 - \frac{(s+s^{-1})^2 (y-2)^2 S^2_{m-1}(y)}{2-s^2-s^{-2} + (y-s^2-s^{-2})(y-2)S^2_{m-1}(y)}.$$
\end{proposition}

\begin{proof}
We have $\rho(w^n) = \left[ \begin{array}{cc}
S_n(z)-W_{22} \, S_{n-1}(z) & W_{12} \, S_{n-1}(z) \\
(2-y)W_{12} \, S_{n-1}(z) & S_n(z)-W_{11} \, S_{n-1}(z) \end{array} \right]$. Similarly,
$$\rho(\overleftarrow{w}^n) = \left[ \begin{array}{cc}
S_n(z)-\overleftarrow{W}_{22} \, S_{n-1}(z) & \overleftarrow{W}_{12} \, S_{n-1}(z) \\
(2-y)\overleftarrow{W}_{12} \, S_{n-1}(z) & S_n(z)-\overleftarrow{W}_{11} \, S_{n-1}(z) \end{array} \right]$$
where
\begin{eqnarray*}
\overleftarrow{W}_{11} &=& S^2_{m}(y) - 2S_m(y)S_{m-1}(y) + (1+2s^2-s^2y)S^2_{m-1}(y),\\
\overleftarrow{W}_{12} &=& (s-s^{-1})S_m(y)S_{m-1}(y) + (s^{-1}+s-sy)S^2_{m-1}(y),\\
\overleftarrow{W}_{22} &=& S^2_{m}(y) + (2-2y)S_m(y)S_{m-1}(y) + (1+2s^{-2}-2y-s^{-2}y+y^2)S^2_{m-1}(y).
\end{eqnarray*}
By a direct calculation, using $S^2_{m}(y) - yS_m(y)S_{m-1}(y) + S^2_{m-1}(y)=1$, we have
\begin{eqnarray*}
\tr \rho(\lambda) &=& \tr (\rho(\overleftarrow{w}^n)\rho(w)) \\
&=& 2 S^2_{n}(z) - 2 \big\{ 2 + (y-2)(y-s^2-s^{-2})S^2_{m-1}(y) \big\} S_{n}(z) S_{n-1}(z) \\
&& + \, \big\{ 2 - (s+s^{-1})^2 (y-2)^2 (y-s^2-s^{-2})S^4_{m-1}(y) \big\} S^2_{n-1}(z)\\
&=& 2 - (s+s^{-1})^2 (y-2)^2 (y-s^2-s^{-2})S^4_{m-1}(y) S^2_{n-1}(z).
\end{eqnarray*}
The lemma then follows from Lemma \ref{S^2}. 
\end{proof}

\section{Twisted Alexander polynomial and Reidemeister torsion} 

\label{section-R}

In this section we briefly review the twisted Alexander polynomial and the Reidemeister torsion of a knot. For more details, see \cite{Li, Wada94-1, FV,Mo-15,Jo, Mi66, Tu}.

\subsection{Twisted Alexander polynomial of a knot}

Let $L$ be a knot in $S^3$. We choose a Wirtinger presentation for the knot group of $L$:
$$
\pi_1(L)=
\langle a_1,\ldots,a_l~|~r_1,\ldots,r_{l-1}\rangle.
$$
The abelianization homomorphism 
$f:\pi_1(L)\to H_1(S^3 \setminus L;\BZ)
\cong {\BZ}
=\langle t
\rangle$ 
is given by 
$f(a_1)=\cdots=f(a_l)=t$. 
Here 
we specify a generator $t$ of $H_1(S^3\backslash K;\BZ)$ 
and denote the sum in $\BZ$ multiplicatively. 

Let $\rho: \pi_1(L) \to SL_2(\BC)$ be a representation. The maps $\rho$ and $f$ naturally induce two ring homomorphisms 
$\tilde{\rho}: {\BZ}[\pi_1(L)] \rightarrow M_2({\BC})$ 
and $\tilde{f}:{\BZ}[\pi_1(L)]\rightarrow {\BZ}[t^{\pm1}]$ respectively, 
where ${\BZ}[\pi_1(L)]$ is the group ring of $\pi_1(L)$ 
and 
$M_2({\BC})$ is the matrix algebra of degree $2$ over ${\BC}$. 
Then 
$\Phi : = \tilde{\rho}\otimes\tilde{f}$ 
defines a ring homomorphism 
${\BZ}[\pi_1(L)]\to M_2\left({\BC}[t^{\pm1}]\right)$. 

Consider the $(l-1)\times l$ matrix $A$ 
whose $(i,j)$-component is the $2\times 2$ matrix 
$$
\Phi\left(\frac{\partial r_i}{\partial a_j}\right)
\in M_2\left({\BZ}[t^{\pm1}]\right),
$$
where 
${\partial}/{\partial a}$ 
denotes the Fox's free calculus. 
For 
$1\leq j\leq l$, 
denote by $A_j$ 
the $(l-1)\times(l-1)$ matrix obtained from $A$ 
by removing the $j$th column. 
We regard $A_j$ as 
a $2(l-1)\times 2(l-1)$ matrix with coefficients in 
$\BC[t^{\pm1}]$. Then Wada's twisted Alexander polynomial \cite{Wada94-1}
of a knot $L$ 
associated to a representation $\rho:\pi_1(L)\to SL_2({\BC})$ 
is defined to be 
$$
\Delta_{L,\rho}(t)
=\frac{\det A_j}{\det\Phi(a_j-1)} .
$$
Note that $\Delta_{L,\rho}(t)$ is well-defined 
up to a factor $t^{2k}~(k\in{\BZ})$. 

\subsection{Torsion of a chain complex}

Let $C$ be a chain complex of finite dimensional vector spaces over $\BC$:
$$
C = \left( 0 \to C_m \stackrel{\partial_m}{\longrightarrow} C_{m-1} 
\stackrel{\partial_{m-1}}{\longrightarrow} \cdots \stackrel{\partial_{2}}{\longrightarrow} 
C_1 \stackrel{\partial_{1}}{\longrightarrow} C_0 \to 0\right)
$$
such that for each $i=0,1, \cdots, m$ the followings hold
\begin{itemize}
\item the homology group $H_i(C)$ is trivial, and
\item a preferred basis $c_i$ of $C_i$ is given.
\end{itemize}

Let $B_i\subset C_i$ be the image of $\partial_{i+1}$. For each $i$ 
choose a basis $b_i$ of $B_i$. The short exact sequence of $\BC$-vector spaces
$$
0 \to B_{i} \longrightarrow C_i \stackrel{\partial_i}{\longrightarrow} B_{i-1} \to 0
$$
implies that  a new basis of $C_i$ can be obtained by taking the union of the vectors of $b_i$
and some lifts $\tilde{b}_{i-1}$ of the vectors $b_{i-1}$. Define $[(b_i \cup \tilde{b}_{i-1})/c_i]$ to be the determinant of the matrix expressing $(b_i \cup \tilde{b}_{i-1})$ in the basis $c_i$. Note that this scalar does not depend on the choice of the lift $\tilde{b}_{i-1}$ of $b_{i-1}$.

The torsion of $C$ is defined to be
$$
\tau(C) := \prod_{i=0}^m \ [(b_i \cup \tilde{b}_{i-1})/c_i]^{(-1)^{i+1}} \ \in \BC\setminus\{0\}.
$$

\begin{remark} Once a preferred basis of $C$ is given, the torsion $\tau(C)$ is independent of the choice of $b_0,\dots,b_m$.
\end{remark}

\subsection{Reidemeister torsion of a CW-complex} Let $M$ be a finite CW-complex and $\rho: \pi_1(M) \to SL_2(\BC)$ a representation. Denote by $\tilde{M}$ the universal covering of $M$. The fundamental group $\pi_1(M)$ acts on $\tilde{M}$ as deck transformations. Then the chain complex $C(\tilde{M}; \BZ)$ has the structure of a chain complex of left $\BZ[\pi_1(M)]$-modules. 

Let $V$ be the 2-dimensional vector space $\BC^2$ with the canonical basis $\{e_1, e_2\}$. Using the representation $\rho$, $V$ has the structure of a right $\BZ[\pi_1(M)]$-module which we denote by $V_{\rho}$. Define the chain complex $C(M; V_{\rho})$ to be $C(\tilde{M}; \BZ) \otimes_{\BZ[\pi_1(M)]} V_{\rho}$, and choose a preferred basis of $C(M; V_{\rho})$ as follows. Let $\{u^i_1, \cdots, u^i_{m_i}\}$ be the set of $i$-cells of $M$, and choose a lift $\tilde{u}^i_j$ of each cell. Then 
$\{ \tilde{u}^i_1 \otimes e_1, \tilde{u}^i_1 \otimes e_2, \cdots, \tilde{u}^i_{m_i} \otimes e_1, \tilde{u}^i_{m_i} \otimes e_2\}$ is chosen to be the preferred basis of $C_i(M; V_{\rho})$.

The Reidemeister torsion $\tau_{\rho}(M)$ is defined as follows:
$$\tau_{\rho}(M)=\begin{cases} \tau(C(M; V_{\rho})) &\mbox{if } \rho \mbox{ is acyclic}, \\ 
0 & \mbox{otherwise}. \end{cases} $$
Here a representation $\rho$ is called acyclic if all the homology groups $H_i(M; V_{\rho})$ are trivial. 

For a knot $L$ in $S^3$ and a representation $\rho: \pi_1(L) \to SL_2(\BC)$, the Reidemeister torsion $\tau_{\rho}(L)$ of $L$ is defined to be that of the knot complement $S^3 \setminus L$. 

The following result which relates the Reidemeister torsion and the twisted Alexander polynomial of a knot is due to Johnson.

\begin{theorem} \cite{Jo} 
\label{Johnson}
Let $\rho: \pi_1(L) \to SL_2(\BC)$ be a representation such that $\det(\rho(\mu)-I) \not= 0$, where $\mu$ is a meridian of $L$. Then the Reidemeister torsion of $L$ is given by
$$
\tau_\rho(L)
=\Delta_{L,\rho}(1). 
$$
\end{theorem}

\section{Proof of main results}

\label{app}

\subsection{Proof of Theorem \ref{main}}

Recall that $K=J(2m,2n)$ and $\pi_1(K) = \la a,b~|~w^na=bw^n \ra$, where $a,b$ are meridians and $w=(ba^{-1})^m(b^{-1}a)^m$. 

Let $r=w^naw^{-n}b^{-1}$. We have $\Delta_{K,\rho}(t)=\det \Phi \left( \frac{\partial r}{\partial a} \right) \big/ \det \Phi(b-1)$. It is easy to see that $\det \Phi(b-1) = t^2-t(s+s^{-1})+1 = t^2-tx+1$.

For an integer $k$ and a word $u$ (in 2 letters $a,b$), let $\delta_k(u)=1+u+\cdots+u^{k}$. The following lemma follows from direct calculations.

\begin{lemma} \label{r/a}
One has 
$$\frac{\partial r}{\partial a} = w^n \left( 1+(1-a)\delta_{n-1}(w^{-1}) w^{-1} \frac{\partial w}{\partial a}\right)$$
where
$$w^{-1} \frac{\partial w}{\partial a} = (a^{-1}b)^m \big( \delta_{m-1}(b^{-1}a) b^{-1} - \delta_{m-1}(ab^{-1}) \big).$$
\end{lemma}

Let 
\begin{eqnarray*}
\Omega_1 &=&\rho \big( \delta_{n-1}(w^{-1} )(a^{-1}b)^m \big),\\
\Omega_2 &=& \big\{t^{-1} \rho \big( \delta_{m-1}(b^{-1}a) b^{-1}\big) - \rho \big( \delta_{m-1}(ab^{-1}) \big) \big\} \big( I-t\rho(a) \big).
\end{eqnarray*}
Then by Lemma \ref{r/a} we have $$\det \Phi \left( \frac{\partial r}{\partial a} \right) = \det (I + \Omega_1 \Omega_2) = 1 + \tr (\Omega_1 \Omega_2) + \det (\Omega_1 \Omega_2).$$

\begin{lemma} \label{O1}
One has $$\Omega_1 = \left[ \begin{array}{cc}
\beta P_{n-1}(z) -  \gamma P_{n-2}(z) & - S_{m-1}(y) \big( s^{-1} P_{n-1}(z) - s P_{n-2}(z) \big) \\
(2-y) S_{m-1}(y) \big( s P_{n-1}(z) - s^{-1} P_{n-2}(z) \big) & \gamma P_{n-1}(z) - \beta P_{n-2}(z) \end{array} \right]$$
where $\beta = S_m(y) - S_{m-1}(y)$ and $\gamma = S_m(y) - (y-1)S_{m-1}(y)$.
\end{lemma}

\begin{proof} 
By Lemma \ref{formulas} we have $$(a^{-1}b)^m = \left[ \begin{array}{cc}
S_{m}(y) -  S_{m-1}(y) & -s^{-1} S_{m-1}(y) \\
-s(y-2) S_{m-1}(y) & S_{m}(y) - (y-1) S_{m-1}(y) \end{array} \right]$$
and $$\delta_{n-1}(w^{-1}) = \left[ \begin{array}{cc}
P_{n-1}(z) -  W_{11} \, P_{n-2}(z) & - W_{12} \, P_{n-2}(z) \\
(y-2)W_{12} \, P_{n-2}(z) & P_{n-1}(z) -  W_{22} \, P_{n-2}(z) \end{array} \right].$$
The lemma then follows by a direct calculation.
\end{proof}

\begin{lemma} \label{O2}
One has \footnotesize
$$\Omega_2 = \left[ \begin{array}{cc}
(st+s^{-1}t^{-1}-2) \big( P_{m-1}(y) -  P_{m-2}(y) \big) & (t-s^{-1})P_{m-1}(y)+(t^{-1}-s) P_{m-2}(y)\\
(2-y)(st-1) \big( t^{-1}P_{m-1}(y) - s^{-1}  P_{m-2}(y) \big) & (s^{-1}t+st^{-1}-y) \big( P_{m-1}(y) - P_{m-2}(y) \big) \end{array} \right].$$
\normalsize
Moreover
$$\det \Omega_2=(t+t^{-1}-x)^2 \left( \frac{S_{m}(y)-S_{m-2}(y)-2}{y-2} \right).$$
\end{lemma}

\begin{proof} By Lemma \ref{formulas} we have
\begin{eqnarray*}
\rho \big( \delta_{m-1}(ab^{-1}) \big) &=& \left[ \begin{array}{cc}
P_{m-1}(y) -  P_{m-2}(y) & s P_{m-2}(y)\\
s^{-1}(y-2) P_{m-2}(y) & P_{m-1}(y) - (y-1) P_{m-2}(y) \end{array} \right],\\
\rho \big( \delta_{m-1}(b^{-1}a) \big) &=& \left[ \begin{array}{cc}
P_{m-1}(y) - (y-1) P_{m-2}(y) & s^{-1} P_{m-2}(y)\\
s(y-2) P_{m-2}(y) & P_{m-1}(y) -  P_{m-2}(y) \end{array} \right].
\end{eqnarray*}
The formula for $\Omega_2$ then follows by a direct calculation. The one for $\det \Omega_2$ is obtained by using the formula $P_k(y) = \frac{S_{k+1}(y)-S_{k}(y)-1}{y-2}$ in Lemma \ref{P_k}.
\end{proof}

We now complete the proof of Theorem \ref{main} by computing the determinant and the trace of the matrix $\Omega_{1}\Omega_2$. By Lemma \ref{formulas} we have $\det\Omega_1 = \frac{S_{n}(z)-S_{n-2}(z)-2}{z-2}$. Hence
\begin{equation} \label{det}
\det(\Omega_1\Omega_2)=(t+t^{-1}-x)^2 \left( \frac{S_{n}(z)-S_{n-2}(z)-2}{z-2} \right)\left( \frac{S_{m}(y)-S_{m-2}(y)-2}{y-2} \right).\end{equation}

By a direct calcultion, using the matrix forms of $\Omega_1$ and $\Omega_2$ in Lemmas \ref{O1} and \ref{O2} and the formula $P_k(y) = \frac{S_{k+1}(y)-S_{k}(y)-1}{y-2}$, we have
\begin{eqnarray*}
\tr (\Omega_1\Omega_2) &=& \big\{ (t+t^{-1})x - x^2 +(x^2-2-y) \big( S_m(y) - (y-1)S_{m-1}(y) \big)\big\} \\
&& \times \, S_{m-1}(y) ( P_{n-1}(z) - P_{n-2}(z))+ (2-y)(x^2-2-y) S^2_{m-1}(y) P_{n-2}(z)\\
 &=& \big\{ (t+t^{-1})x - x^2 +(x^2-2-y)\big( S_m(y) - (y-1)S_{m-1}(y) \big)\big\} \\
 && \times \, S_{m-1}(y) S_{n-1}(z)+(z-2) P_{n-2}(z)\\
  &=& \big\{ (t+t^{-1})x - x^2 +(x^2-2-y\big( S_m(y) - (y-1)S_{m-1}(y) \big)\big\}\\
  && \times \, S_{m-1}(y) S_{n-1}(z)+S_{n-1}(z)-S_{n-2}(z)-1.
\end{eqnarray*}
Since $S_{n-2}(z)= \big\{ 1 - (y+2-x^2) S_{m-1}(y) (S_{m-1}(y)- S_{m-2}(y)) \big\} S_{n-1}(z)$ we get
\begin{equation} \label{trace}
\tr (\Omega_1\Omega_2) = \big( (t+t^{-1})x - x^2 \big) S_{m-1}(y) S_{n-1}(z)-1.
\end{equation}

Finally, by combining the equations \eqref{det}, \eqref{trace} and $\Delta_{K,\rho}(t) = \frac{1+\tr (\Omega_1\Omega_2)+\det (\Omega_1\Omega_2)}{t^2-tx+1}$ we complete the proof of Theorem \ref{main}.

\subsection{Proof of Theorem \ref{main1}} Note that $\det \big( \rho(b)-I \big)=2-x$. Since $\tau_{\rho}(K)= \Delta_{K,\rho}(1)$ for $x \not= 2$, Theorem \ref{main1} follows directly from Theorem \ref{main}.

\subsection{Proof of Theorem \ref{main2}} Let $M$ be the 3-manifold obtained by a $\frac{p}{q}$-surgery on the genus one two-bridge knot $K=J(2m,2n)$. Suppose $\rho: \pi_1(K) \to SL_2(\BC)$ is a nonabelian representation which extends to a representation $\rho: \pi_1(M) \to SL_2(\BC)$. Recall that $\lambda$ is the canonical longitude corresponding to the meridian $\mu=a$. If $\tr \rho(\lambda) \not= 2$, then by \cite{Ki2015} (see also \cite{Ki1994-fibered, Ki1994-8}) the Reidemeister torsion of $M$ is given by   
\begin{equation}
\label{Dehn}
\tau_{\rho}(M) = \frac{\tau_{\rho}(K)}{2-\tr \rho(\lambda)}.
\end{equation}

By Theorem \ref{main1} we have $$\tau_{\rho}(K)=
(2-x) \left( \frac{S_{m}(y)-S_{m-2}(y)-2}{y-2} \right) \left( \frac{S_{n}(z)-S_{n-2}(z)-2}{z-2} \right) + xS_{m-1}(y) S_{n-1}(z)$$ if $x \not=2$. By Proposition \ref{longitude} we have $$\tr \rho(\lambda) - 2 = - \frac{x^2 (y-2)^2 S^2_{m-1}(y)}{4-x^2 + (y+2-x^2)(y-2)S^2_{m-1}(y)}.$$ By Lemma \ref{S^2} we have $S_{m-1}(y) \not=0$. This implies that $\tr \rho(\lambda) \not= 2$ if and only if $x \not= 0$. Theorem \ref{main2} then follows from \eqref{Dehn}.

\end{document}